\documentclass[a4paper, notitlepage, 12pt]{amsart}

\usepackage[english]{babel}
\usepackage[applemac]{inputenc}

\usepackage{amscd, amsfonts, amsmath, amssymb, amsthm, array, bm, graphicx, manfnt, dsfont, geometry, bm}
\usepackage{mathsymbols, maththeorems}
\usepackage[pdftex]{hyperref}
\usepackage[all,cmtip]{xy}

\usepackage{maththeorems}
\usepackage{mathsymbols}
\newcommand{\vol}{\mathrm{vol}}

\DeclareMathOperator*{\diam}{diam}
\DeclareMathOperator*{\Capacity}{Cap}

\begin{document}


\title{Heat content asymptotics of some random Koch type snowflakes}

\author{Philippe H.\ A.\ Charmoy}

\address{Mathematical Institute, Andrew Wiles Building, Radcliffe Observatory Quarter, Oxford OX1 6GG, United Kingdom}
\email{charmoy@maths.ox.ac.uk}
\date{\today.}

\maketitle


\begin{abstract}
We consider the short time asymptotics of the heat content $E$ of a domain $D$ of $\R^d$. The novelty of this paper is that we consider the situation where $D$ is a domain whose boundary $\partial D$ is a random Koch type curve.

When $\partial D$ is spatially homogeneous, we show that we can recover the lower and upper Minkowski dimensions of $\partial D$ from the short time behaviour of $E(s)$. Furthermore, in some situations where the Minkowski dimension exists, finer geometric fluctuations can be recovered and the heat content is controlled by $s^\alpha e^{f(\log(1/s))}$ for small $s$, for some $\alpha \in (0, \infty)$ and some regularly varying function $f$. The function $f$ is not constant is general and carries some geometric information.

When $\partial D$ is statistically self-similar, then the Minkowski dimension and content of $\partial D$ typically exist and can be recovered from $E(s)$. Furthermore, the heat content has an almost sure expansion $E(s) = c s^{\alpha} N_\infty + o(s^\alpha)$ for small $s$, for some $c$ and $\alpha \in (0, \infty)$ and some positive random variable $N_\infty$ with unit expectation arising as the limit of some martingale.
\end{abstract}


 



\section{Introduction}

Let $D$ be a bounded domain in $\R^d$ and $\Delta$ be the Dirichlet Laplacian. Then the spectrum $\Lambda$ of $-\frac 12 \Delta$ is discrete and of the form
$$
0 < \lambda_1 \leq \lambda_2 \leq \cdots,
$$
where the eigenvalues are repeated according to their multiplicity. Interest in the geometric information about $D$ contained in $\Lambda$ started more than 100 years ago and was crystallised by Kac in his paper \cite{Kac1966} entitled \emph{Can one hear the shape of a drum?} In other words: Are isospectral domains always isometric? The answer is no in general, as shown, for example, in \cite{Buseretal1994, Gordonetal1992, Milnor1964}. But it is natural to enquire how much information about the geometry of $D$ is encoded by $\Lambda$, and, when $D$ is a random domain, how much of its distribution can be recovered.

As the spectral decomposition of the heat kernel with absorption on the boundary $p_D$ indicates, the heat content
$$
E_D(s) = \int_D \left(1- \int_D p_D(s,x,y) dy \right) dx
$$
provides a natural proxy for the eigenvalues of $- \frac 12 \Delta$. Recall that $E$ may also be expressed more intuitively as
\begin{equation}\label{eq::defHeatContent}
E_D(s) = \int_D u_D(s,x) dx,
\end{equation}
where $u_D$ is the solution to the heat equation with unit Dirichlet boundary condition and zero initial condition, i.e.\ is the solution to
\begin{equation}\label{eq::heatEquationSpecProblems}
\begin{aligned}
	\partial_s u_D(s,x) &= \frac 12 \Delta u_D(s,x), &
	(s, x) \in (0, \infty) \times D,\\
	u_D(0, x) & = u_D(0+,x) = 0, &
	x \in D,\\
	u_D(-, x) & = 1,&
	x \times \partial D.
\end{aligned}
\end{equation}
We will omit the dependence on $D$ from the notation when there is no risk of confusion.

The heat content presents the advantage that it is amenable to probabilistic techniques. Furthermore, it is a convenient object to recover information about the geometry of the boundary of $D$. Interest in this naturally intensified when Berry studied the spectral properties of domains with a fractal boundary in \cite{Berry1979, Berry1980} and conjectured that the Hausdorff dimension of $\partial D$ should be encoded by $\Lambda$. This was first disproved by \cite{BC1986} who showed that the Minkowski dimension was the relevant measure of roughness.

Since then, the short time asymptotics of the heat content have been studied extensively. Planar domains with polygonal boundary are discussed in \cite{vdBS1988, vdBS1990}. Some domains with fractal boundary, including the triadic Koch snowflake, are discussed in \cite{FLV1994, LV1996, vdBdH1999}. In \cite{vdB1994}, van den Berg proved that, under some regularity conditions, if the Minkowski dimension $\gamma$ of $\partial D$ exists, then
$$
E(s) \asymp s^{(d-\gamma)/2}
$$
for small $s$, where $f(x) \asymp g(x)$ means that $c^{-1} f(x) \leq g(x) \leq c f(x)$ for some $c \in (0, \infty)$; this is the heat content analogue of the results in \cite{BC1986}.

Here, we study the heat content asymptotics of two families of Koch type snow\-flakes, thereby addressing a situation closely related to \cite{FLV1994}. First, we discuss scale homogeneous snowflakes whose  boundary is constructed from a sequence $(\xi_n, n \in \N)$ of natural numbers determining the size and number of linear pieces used at each iteration of the construction throughout the set. This is related to the Sierpinski gaskets discussed in \cite{Hambly1992, BH1997}. Second, we discuss snowflakes whose boundary is statistically self-similar, and so we use the theory of general branching processes to study their geometry and heat content.

The paper is organised as follows. In Section \ref{sec::heatContentEstimates}, we discuss estimates for the heat content along the lines of \cite{vdB1994} and use this to show that one can recover the lower and upper Minkowski dimension of $\partial D$ from the heat content in the following way.

\begin{theorem}
	Let $D$ be a bounded domain in $\R^d$. Write $\alpha$, respectively $\beta$, for the lower, respectively upper, Minkowski dimension of $\partial D$. Then, under the regularity conditions given in Assumptions \ref{ass::regularityBrownianMotion} and \ref{ass::capacitaryDensity},
	$$
	\liminf_{s \to 0} \left(\frac d2 - \frac{\log E(s)}{\log s} \right) = \frac{\alpha}{2} \quad \text{and} \quad \limsup_{s \to 0} \left(\frac d2 - \frac{\log E(s)}{\log s} \right) = \frac{\beta}{2}.
	$$
\end{theorem}

Intuitively, this means that \emph{one can hear the lower and upper Minkowski dimensions of $\partial D$}, and in particular determine whether they are equal.

In Section \ref{sec::scaleHomogeneous}, we detail the construction of the family of scale homogeneous snow\-flakes mentioned above. We then look at their Minkowski dimension and content and use that to study their heat content. In particular, we show that when the sequence $(\xi_n, n \in \N)$ used to build the snowflake is stationary and ergodic, the Minkowski dimension exists, and discuss how to construct examples where it does not. Furthermore, we show that the rate of convergence in the ergodic theorem dictates the short time asymptotics of the heat content. An important example is when $(\xi_n, n \in \N)$ is i.i.d., in which case we use the law of the iterated logarithm to prove the following theorem.

\begin{theorem}
	Let $D$ be a scale homogeneous snowflake constructed with an i.i.d.\ sequence. Then the Minkowski dimension $\gamma$ of the boundary of the snowflake exists almost surely. Furthermore, under the regularity conditions given in Assumptions \ref{ass::regularityBrownianMotion} and \ref{ass::capacitaryDensity}, for some positive constants $c_1, \dots, c_6$, we have
$$
c_1 s^{1- \gamma/2} e^{-c_2 \psi(1/s)} \leq E(s) \leq c_3 s^{1- \gamma/2} e^{c_4 \psi(1/s)}
$$
for small $s$, while
$$
\liminf_{s \to 0} \frac{E(s) e^{c_5 \psi(1/s)}}{s^{1-\gamma/2}} < \infty \quad \text{and} \quad \limsup_{s \to 0} \frac{E(s) e^{-c_6 \psi(1/s)}}{s^{1-\gamma/2}} > 0,
$$
where
$$
\psi(x) = \sqrt{\log x \log \log \log x}.
$$
\end{theorem}

The function $\psi$ is in some sense the best possible, and this result intuitively implies that \emph{one can hear the law of the iterated logarithm}.

In Section \ref{sec::gbp}, we give a brief introduction to the theory of general branching processes and introduce the notation necessary to discuss statistically self-similar sets.

Finally, in Section \ref{sec::selfSimilar}, we detail the construction of our statistically self-similar snowflakes. Using the theory of general branching processes, we show that the Minkowski dimension and content of the boundary of these snowflakes typically exist and that this implies the following result for the heat content.

\begin{theorem}
	Under the regularity conditions given in Assumptions \ref{ass::regularityBrownianMotion} and \ref{ass::capacitaryDensity}, the heat content of the statistically self-similar snowflakes satisfies
	$$
	s^{-(1-\gamma/2)} E(s) \to c_7 N_\infty, \text{ a.s.\ and in } L^1,
	$$
	as $s \to 0$, for some positive constant $c_7$ and some positive random variable $N_\infty$ with unit expectation arising as the limit of some martingale.
\end{theorem}

Together with a similar result for the geometry of the boundary of the snowflake, this result has the intuitive interpretation that \emph{one can hear the Minkowski dimension and content of the boundary}.


\subsection*{Notation}

Throughout the document, the symbol $c_i$ with $i \in \N$ will mean \emph{some positive constant} whose value is typically fixed for the length of a proof or a subsection.


\section{Heat content estimates}

\label{sec::heatContentEstimates}

In this section, we start by recalling the definition of inner Minkowski dimension and content and then derive bounds on the heat content in a fashion inspired by \cite{BC1986, vdB1994}. Finally, we look at an example.


\subsection{Inner Minkowski dimension and content}

\label{subsec::dimensionAndContentDefinition}

Let $K$ be a bounded subset of $\R^d$. The \emph{$\epsilon$-neighbourhood} of $K$ is defined as
$$
K_\epsilon = \{x \in \R^d : d(x, K) \leq \epsilon\},
$$
where $d(x, A)$ is the Euclidean distance between the point $x$ and the set $A$.

For a bounded domain $D$ of $\R^d$, we call $(\partial D)_\epsilon \cap D$ the \emph{inner $\epsilon$-tubular neighbourhood} and will write $\mu_D(\epsilon)$ for the volume of that neighbourhood, i.e.\
$$
\mu_D(\epsilon) = \vol_d((\partial D)_\epsilon \cap D),
$$
where $\vol_d$ denotes the Lebesgue measure on $\R^d$; again, we will omit the dependence on $D$ when there can be no confusion. We call \emph{inner lower, respectively upper, Minkowski dimension} of $\partial D$ the quantity
$$
\underline{\dim}_M \partial D = d - \limsup_{\epsilon \to 0} \frac{ \log \mu(\epsilon)}{\log \epsilon}, \quad \text{respectively} \quad \overline{\dim}_M \partial D = d - \liminf_{\epsilon \to 0} \frac{\log \mu(\epsilon)}{ \log \epsilon}.
$$
When these two quantities are equal, we say that the inner Minkowski dimension exists and use the notation $\dim_M \partial D$ instead. In all cases of interest here, the inner Minkowski dimension is equal to the usual Minkowski dimension defined, for example, in \cite{Falconer1986a}.

When the Minkowski dimension of $\partial D$ exists, we define the \emph{inner lower, respectively upper, Minkowski content} as
$$
\cM_* = \liminf_{\epsilon \to 0} \epsilon^{\dim_M \partial D - d} \mu(\epsilon), \quad  \text{respectively} \quad \cM^* = \limsup_{\epsilon \to 0} \epsilon^{\dim_M \partial D- d} \mu(\epsilon).
$$
When these two quantities are equal, we say that the inner Minkowski content exists and use the notation $\cM$ instead.


\subsection{The estimates}

It will be convenient to be able to solve the heat equation under the conditions in \eqref{eq::heatEquationSpecProblems} using the probabilistic representation
\begin{equation}\label{eq::probabilisticSolution}
u(s,x) = \bP_x(T_{D^c} \leq s),
\end{equation}
where $T_{D^c}$ is the hitting time of $D^c$ of Brownian motion and $\bP_x$ is the law of Brownian motion started at $x$. Therefore, we will always make the following assumption.

\begin{assumption}\label{ass::regularityBrownianMotion}
All	the points of $\partial D$ are regular for $D^c$.
\end{assumption}

Recall that this assumption is always satisfied for simply connected planar domains; see Proposition II.1.14 of \cite{Bass1995}. This will cover all the examples discussed here.

We now prove a first upper bound for the heat content following the argument of \cite{vdB1994}.

\begin{theorem}\label{thm::upperBoundHeatContent}
Let $D$ be a bounded domain in $\R^d$ and let $\omega: \R_+ \to \R_+$ be an increasing function with $\omega(0) = 0$. Then, for every $s \geq 0$,
$$
E(s) \leq \mu(\omega(s)) + 2^{(d+2)/2} \vol_d(D)e^{-\omega(s)^2/4s}.
$$	
\end{theorem}

\begin{proof}
Let $s \geq 0$ and put
	$$
	A_s = \{x \in D: d(x, \partial D) \leq \omega(s)\}.
	$$
By \eqref{eq::probabilisticSolution}, we have
	\begin{equation}\label{eq::firstStepBoundHeatContent}
	\begin{aligned}
	E(s) 	& = \int_{A_s} \bP_x(T_{D^c} \leq s) dx + \int_{D\setminus A_s} \bP_x(T_{D^c} \leq s) dx \\
			& \leq \mu(\omega(s)) + \int_{D\setminus A_s} \bP_x(T_{D^c} \leq s) dx.
	\end{aligned}
	\end{equation}

But now, by Lévy's maximal inequality, e.g.\ Theorem 3.6.5 of \cite{Simon2005},
	\begin{align*}
		\bP_x(T_{D^c} \leq s) & \leq \bP_x(T_{\cB(x, d(x, \partial D))^c} \leq s)\\
		& = \bP_0 \left( \sup_{0 \leq u \leq s} |B(u)| \geq d(x, \partial D) \right)\\
		& \leq 2 \bP_0 (|B(s)| \geq d(x, \partial D))\\
		& = 2 (2 \pi s)^{-d/2} \int_{|y| \geq d(x, \partial D)} e^{-|y|^2/2s} dy,
	\end{align*}
	where $\cB(x, \delta)$ denotes the open ball of radius $\delta$ centred at $x$. Changing variables and standard Gaussian estimates then yield
	\begin{equation}\label{eq::gaussianEstimate}
	\bP_x(T_{D^c} \leq s) \leq 2^{(d+2)/2} e^{-d(x, \partial D)^2/4s}.
	\end{equation}
	Since $d(x, \partial D) > \omega(s)$ for $x \in D \setminus A_s$, using this estimate in \eqref{eq::firstStepBoundHeatContent} completes the proof.
\end{proof}

The other estimates that we will use are those of Theorems 1.2 to 1.4 of \cite{vdB1994} which we recall here for convenience.

\begin{theorem}[van den Berg]\label{thm::vdBUpperBound}
Let $D$ be a bounded domain in $\R^d$. Then, for every $s \geq 0$,
$$
E(s) \leq 2^{d/2} s^{-1} \int_0^\infty \epsilon e^{-\epsilon^2/4s} \mu(\epsilon) d\epsilon.
$$	
\end{theorem}

The lower bound for the heat content is proved under a capacitary condition that we state now; see \cite{BC1986, vdB1994} for more background. We write $\Capacity (A)$ for the Newtonian capacity of the set $A$.

\begin{assumption}[Capacitary density]\label{ass::capacitaryDensity}
	For the bounded domain $D$ of $\R^d$ with $d \geq 2$ there exists a positive constant $c$ such that, for all $x \in D$ and $r \in (0, \diam(D))$,
	$$
	\Capacity(B(x, r) \cap \partial D) \geq c \Capacity(B(x, r)).
	$$
\end{assumption}

This assumption is usual in these problems and we say that the \emph{capacitary density of $\partial D$ is bounded below} when it is satisfied. This is always the case when $D$ is a simply connected planar domain; see \cite{vdB1994}.

\begin{theorem}[van den Berg]\label{thm::vdBLowerBound}
	Let $D$ be a bounded domain of $\R^d$ and assume that either $d= 1$ or $d \geq 2$ and the capacitary density of $\partial D$ is bounded below. Then, for every $s \geq 0$,
	$$
	E(s) \geq c_1 \mu(c_2 s^{1/2}).
	$$
\end{theorem}

These results illustrate the role of the Minkowski dimension in establishing a lower bound on $E(s)$ for small $s$. Furthermore, the change of variables $\eta = \epsilon^2/4$ yields
\begin{equation}\label{eq::changeVariables}
E(s) \leq 2^{(d+2)/2} s^{-1} \int_0^\infty e^{-\eta /s} \mu(2 \eta^{1/2}) d \eta.
\end{equation}
If $\mu(\epsilon)$ does not oscillate too much for small $\epsilon$, then an Abelian theorem can be used to deduce the behaviour of $E(s)$ for small $s$ as we show now.

\begin{theorem}\label{thm::heatContentAbelianArgument}
Let $D$ be a bounded domain of $\R^d$ and assume that either $d= 1$ or $d \geq 2$ and the capacitary density of $\partial D$ is bounded below. Assume further that there exist $\gamma \in (0, \infty)$ and a slowly varying function $L$ such that
$$
\mu(\epsilon) \asymp \epsilon^{d-\gamma} L(1/\epsilon)
$$
for small $\epsilon$. Then,
$$
E(s) \asymp s^{(d- \gamma)/2} \tilde L(1/s)
$$
for small $s$, where $\tilde L$ is the slowly varying function defined by
$$
\tilde L(x) = L(x^{1/2}/2).
$$
\end{theorem}

\begin{proof}
	Define
	$$
	\tilde \mu(\eta) = \mu(2 \eta^{1/2}) = 2^{d- \gamma} \eta^{(d - \gamma)/2} \tilde L(1/\eta).
	$$
	Using \eqref{eq::changeVariables}, an integration by parts and applying Theorem XIII.5.3 in \cite{Feller1968}, we have
	$$
	E(s) \leq c_3 s^{-1} \int_{0}^\infty e^{-\eta/s} \tilde \mu(\eta) d \eta = c_3 \int_0^\infty e^{-\eta/s} \tilde \mu(d \eta)\sim c_4 s^{(d- \gamma)/2} \tilde  L(1/s),
	$$
	as $s \to 0$, where the notation $f(x) \sim g(x)$ means that $f(x)/g(x) \to 1$. A similar lower bound follows immediately using Theorem \ref{thm::vdBLowerBound}.
\end{proof}

When $\mu(\epsilon)$ oscillates too wildly, for example when the Minkowski dimension does not exist, the Abelian theorem is not applicable. But we can then rely on Theorem \ref{thm::upperBoundHeatContent} to study the short time asymptotics of the heat content and get the following announced result.

\begin{theorem}\label{thm::lowerAndUpperHeatDimension}
Let $D$ be a bounded open domain of $\R^d$ and assume that either $d= 1$ or $d \geq 2$ and the capacitary density of $\partial D$ is bounded below. Then,
$$
\liminf_{s \to 0} \left(\frac d2 - \frac{\log E(s)}{\log s} \right) = \frac 12 \underline{\dim}_M \partial D
$$
and
$$
\limsup_{s \to 0} \left(\frac d2 - \frac{\log E(s)}{\log s} \right) = \frac 12 \overline{\dim}_M \partial D.
$$
\end{theorem}

\begin{proof}
Put $\alpha = \underline{\dim}_M \partial D$ and $\beta = \overline{\dim}_M \partial D$ and let $\delta \in (0, \infty)$ be small.

By definition of the Minkowski dimension, we have, on the one hand, that there exists a sequence $(\epsilon_n, n \in \N)$ converging to 0 along which
$$
\mu(\epsilon_n) \leq \epsilon_n^{d- (\alpha+ \delta)},
$$
and, on the other hand, that
$$
\mu(\epsilon) \geq \epsilon^{d-(\alpha - \delta)}
$$
for $\epsilon$ small enough.

Setting $\omega(s) = \sqrt{2 d s \log(1/s)}$, which is increasing around 0 and putting $\omega(s_n) = \epsilon_n$ together with Theorems \ref{thm::upperBoundHeatContent} and \ref{thm::vdBLowerBound} shows that
$$
E(s_n) \leq s_n^{(d-(\alpha + \delta))/2} \log(1/s_n)^{(d-(\alpha + \delta))/2} + c_5 s_n^{d/2}.
$$
and
$$
E(s) \geq c_6 s^{(d-(\alpha - \delta))/2}
$$
for $s$ small enough. The result for the liminf follows.

The proof of the limsup is similar, and somewhat simpler, relying on Theorems \ref{thm::vdBUpperBound} and \ref{thm::vdBLowerBound}.
\end{proof}


\subsection{Self-affine boundaries}

We conclude this section with a brief mention of a domain whose boundary's Hausdorff and Minkowski dimensions disagree and where the results discussed here can be applied. This is an alternative example to that of \cite{BC1986} showing that the Minkowski dimension is the relevant measure of roughness for heat conduction problems; in the example presented here, however, the domain is connected.

Following \cite{PS2013}, we define the boundary of the domain using the self-affine carpets of Bedford \cite{Bedford1984} and McMullen \cite{McMullen1984} whose construction we recall briefly now. Let $m < n$ be two integers. Divide the unit square $[0,1]^2$ into $m n$ rectangles of height $m^{-1}$ and width $n^{-1}$. Keep some rectangles according to a pattern $P$ and discard the others. This produces a compact set $K_1$. For each rectangle of the pattern, repeat the procedure. This produces a compact subset $K_2$ of $K_1$. Continue indefinitely to get a compact set
$$
K(P) = \bigcap_{j= 1}^\infty K_j.
$$

A natural way to represent a pattern is to use an $m \times n$ matrix with entries in $\{0,1\}$ where each 1 indicates a rectangle that we choose to keep. The carpet corresponding to pattern $P$ is then
$$
K(P) = \left\{ \sum_{k = 1}^\infty (a_k n^{-k}, b_k m^{-k}) : (a_k, b_k) \in D \right\},
$$
where $D = \{(i,j) : P(j,i) = 1\}$. Here, the rows and columns of $P$ are numbered starting from 0, and the rows are numbered from bottom to top. As an illustration, the first 3 iterations corresponding to the pattern
$$
A = \begin{pmatrix}
	0 & 1 & 1 & 1\\
	1 & 0 & 0 & 0
\end{pmatrix}
$$
are shown in Figure \ref{fig:selfaffinecarpet}.

\begin{figure}
	\begin{center}
		\includegraphics[height = 3cm]{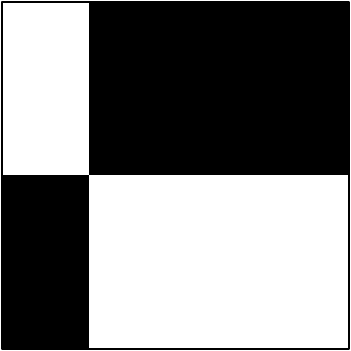}
		\includegraphics[height = 3cm]{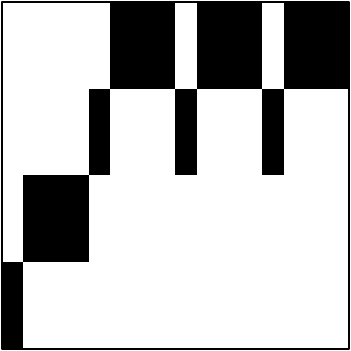}
		\includegraphics[height = 3cm]{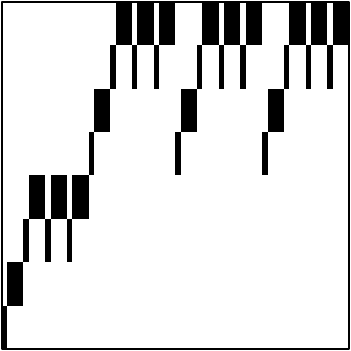}	
	\end{center}
	\caption{First 3 steps of the self-affine carpet for pattern $A$.}
	\label{fig:selfaffinecarpet}
\end{figure}

The Hausdorff dimension of the self-affine carpet built from pattern $P$ is
$$
\log_m \left(\sum_{j=0}^{m-1} r(j)^{\log_n m} \right),
$$
where $r(j)$ is the number of chosen rectangles of the pattern in row $j$, while its Minkowski dimension (which exists) is given by
$$
1 + \log_n \left( \frac 1m \sum_{j=0}^{m-1} r(j) \right);
$$
see \cite{Bedford1984, McMullen1984, P1994}.

Of course, the set $K(P)$ is not a continuous curve in general. But, for appropriate patterns, we can alter the construction of the carpet using the reflected pattern
$$
P^r(i,j) = P(i, n-1-j)
$$
to create a set that is the graph of a continuous function. This can be done for pattern $A$ using a procedure that we describe now.

As in the construction of $K(A)$, start with pattern $A$. Reproduce pattern $A^r$ in the chosen rectangle of column 2 and $A$ in the others. Repeat this procedure for chosen rectangles inside a pattern $A$; and, inside a pattern $A^r$, reproduce pattern $A$ for the chosen rectangle of column 1 and $A^r$ for the others. Continue indefinitely to get a compact set $K_c(A)$; this set is the graph of a continuous function $f_A: [0,1] \to [0,1]$. The first 3 iterations are shown in Figure \ref{fig:contselfaffinefun}. Examining the calculations in \cite{P1994} shows that this alteration of the construction does not change the Hausdorff or the Minkowski dimension; see also \cite{PS2013}.

\begin{figure}
	\begin{center}
		\includegraphics[height = 3cm]{selfAffineLevelOne.pdf}
		\includegraphics[height = 3cm]{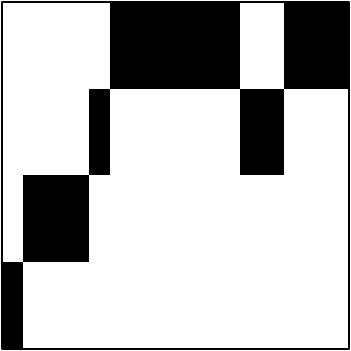}
		\includegraphics[height = 3cm]{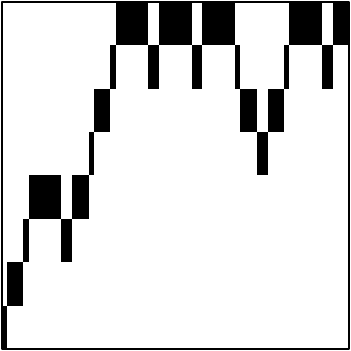}
	\end{center}
	\caption{First 3 steps of the continuous self-affine curve for pattern $A$.}
	\label{fig:contselfaffinefun}
\end{figure}

Now define the continuous function $g: [0,2] \to [1,3]$ by
$$
g(t) = 1 + f(t) \bone_{[0,1]}(t) + (2 - f(2-t))\bone_{(1,2]}(t),
$$
and let $D$ be the simply connected domain inside the Jordan curve defined by
$$
\{(t, g(t)) : t \in [0,2]\} \cup (\{0\}\times[0,1]) \cup ([0,2] \times \{0\})  \cup ( \{2\}\times[0,3]).
$$
Furthermore, put
$$
\nu(\epsilon) = \vol_2(\{ x \in \R^2 : d(x, \{(t,g(t)) : t \in [0,2]\}) \leq \epsilon\}).
$$
Because of the symmetry of $g$, we have
$$
\nu(\epsilon) = 2 \vol_2(\{x \in D :d(x, \{(t,g(t)) : t \in [0,2]\}) \leq \epsilon\}) + \pi \epsilon^2,
$$
and therefore
$$
\nu(\epsilon) = 2 \, \mu (\epsilon) + O(\epsilon^2).
$$
It follows that the inner Minkowski dimension of $\partial D$ is equal to the Minkowski dimension of $K_c(A)$, which is different from its Hausdorff dimension. But because $D$ is simply connected, Theorem \ref{thm::lowerAndUpperHeatDimension} can be used to show that
$$
\lim_{s \to 0} \left(1 - \frac{\log E(s)}{\log s} \right) = \frac 12 + \frac 12 \log_n \left( \frac 1m \sum_{j=0}^{m-1} r(j) \right).
$$


\section{Scale homogeneous snowflakes}

\label{sec::scaleHomogeneous}


\subsection{Construction}

We introduce a family of scale homogeneous random snow\-flakes by generalising the Koch curve. In the construction of the usual, triadic, Koch curve, the segment $[0,1]$ is replaced by a curve $K(1)$, say, made of 4 segments of length $1/3$ arranged to produce a spike in the middle. This procedure is then iterated and produces a limiting self-similar curve.

Here we proceed similarly, but using different building blocks with different numbers of spikes. More precisely, for $a \in A$, a bounded subset of $\N$, put
$$
m(a) = 3 a +1 \quad \text{and} \quad \ell(a) = 2a + 1,
$$
and let $K(a)$ be the curve made of $m(a)$ segments of length $\ell(a)^{-1}$ arranged to produce $a$ spikes as depicted in Figure \ref{fig::buildingBlocksSnowflakes}. Now, let $\xi = (\xi_n, n \in \N)$ be a sequence of elements of $A$. We will write $K(\xi)$ for the Koch curve where we used $K(\xi_n)$ as a building block at iteration $n$. For example, the curve formed from the first iterations for $\xi = (1, 3, 2, 1, \dots)$ is shown in Figure \ref{fig::exampleFlake}.

\begin{figure}
\includegraphics[width = 10 cm]{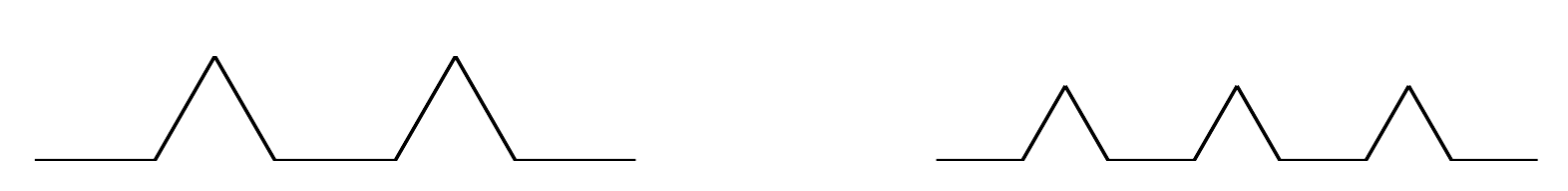}
\caption{The curves $K(2)$ and $K(3)$.}
\label{fig::buildingBlocksSnowflakes}
\end{figure}

\begin{figure}
\includegraphics[width = 10 cm]{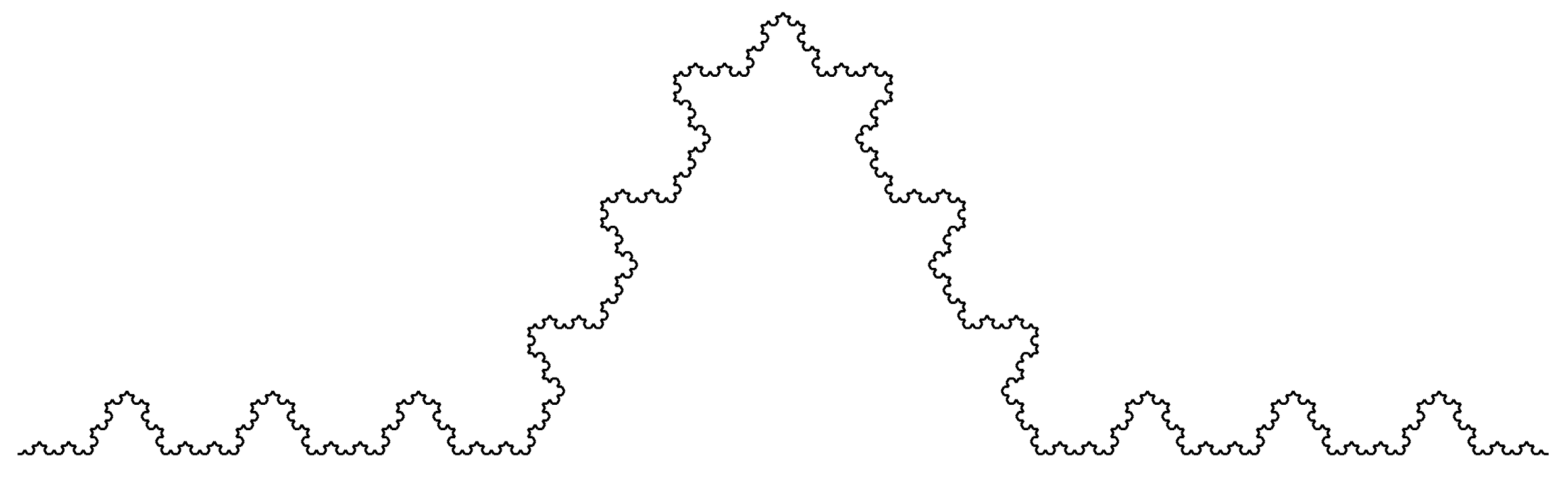}
\caption{First three iterations of $K(1,3,2,1, \dots)$.}
\label{fig::exampleFlake}
\end{figure}

It will be convenient to use the notation
$$
\ell_n = \ell(\xi_n), \quad m_n = m(\xi_n), \quad \epsilon_n^{-1} = L_n = \prod_{i =1}^n \ell_i \quad \text{and}\quad M_n = \prod_{i = 1}^n m_i.
$$
With these definitions, the $n$th iteration in the construction of $K(\xi)$ consists of $M_n$ segments of length $\epsilon_n = L_n^{-1}$.

The domain $D = D(\xi)$ in which we are interested is the one enclosed by the Jordan curve made of three copies of $K(\xi)$ arranged as in the construction of the usual Koch snowflake. In particular, $D$ is simply connected.


\subsection{Fractal dimension and content}

Since $A$ is bounded, we have
\begin{equation}\label{eq::coverinNumber}
\epsilon_n \asymp \epsilon_{n+1} \quad \text{and} \quad M_n \asymp N(\epsilon_n, K(\xi)) \asymp N(\epsilon_{n+1}, K(\xi)),
\end{equation}
where $N(\epsilon, K)$ is the covering number of $K$ by balls of radius $\epsilon$. This observation and the mass distribution principle enable us to calculate the dimension of $K(\xi)$.

\begin{theorem}\label{thm::dimensionScaleHomogeneousFlake}
The Hausdorff and Minkowski dimensions of the set $K(\xi)$ are given by
$$
\dim K(\xi) = \underline{\dim}_M K(\xi) = \liminf_{n \to \infty} \frac{\log M_n}{\log L_n}
$$
and
$$
\overline{\dim}_M K(\xi) = \limsup_{n \to \infty} \frac{\log M_n}{\log L_n}.
$$
\end{theorem}

\begin{proof}
The calculation of the lower and upper Minkowski dimension follows directly from \eqref{eq::coverinNumber}.

To simplify the notation for the calculation of the Hausdorff dimension, put $\alpha = \underline{\dim}_M K(\xi)$. It is standard that $\dim K(\xi) \leq \alpha$. To prove the other inequality, let $\delta \in (0, \infty)$. By \eqref{eq::coverinNumber}, there exists $m$ large such that, for every $n \geq m$, we have $M_n^{-1} \leq \epsilon_n^{\alpha - \delta}$. Furthermore, the set $K(\xi)$ is made of $M_m$ (disjoint up to $m-1$ points) copies of the set
$$
\tilde K = L_m^{-1} K(\xi_{m+1}, \xi_{m+2}, \dots).
$$

Now consider the Borel measure $\nu$ assigning mass $M_n^{-1}$ to sets of size $L_n^{-1}$ in the obvious way for $n \geq m$. Let $U$ be a subset of $\tilde K$ and let $n$ be such that $\epsilon_{n+1} \leq \diam U < \epsilon_n$. Then, we have
$$
\nu(U) \leq M_n^{-1} \leq \epsilon_n^{\alpha- \delta} \leq c \epsilon_{n+1}^{\alpha- \delta} \leq c (\diam U)^{\alpha- \delta},
$$
for some positive constant $c$, thanks to \eqref{eq::coverinNumber}.

By the mass distribution principle, it follows that $\tilde K$, and hence $K(\xi)$, has Hausdorff dimension at least $\alpha - \delta$. Since $\delta$ is arbitrary, this completes the proof.
\end{proof}

To continue the geometric description of the scale homogeneous snowflakes, we now look at the volume of the inner tubular neighbourhoods of $D$. The following result immediately implies that the inner Minkowski dimensions agree with the usual definitions and calculations of Theorem \ref{thm::dimensionScaleHomogeneousFlake}.

\begin{lemma}\label{lem::tubNeighbourhoodScaleHomo}
	The volume of the inner tubular neighbourhoods of $D$ satisfies
	$$
	\mu(\epsilon_n) \asymp M_n L_n^{-2}.
	$$
\end{lemma}

\begin{proof}
	On the one hand, the volume of the $\epsilon_n$ tubular neighbourhood in one third of the snowflake is bounded below by the volume of the level $n$ spike of a third of the boundary, i.e.\ a copy of $K(\xi)$. This means that
	$$
	\frac 13 \mu(\epsilon_n) \geq \frac{\sqrt{3}}{4} M_{n-1} \xi_n \epsilon_n^2 \asymp M_n \epsilon_n^2,
	$$
	where we used that $A$ is bounded.
	
	On the other hand, this same quantity is bounded above by the volume of the set
	$$
	\{x \in \R^d : d(x, K(\xi_1, \dots, \xi_n)) \leq \epsilon_n\},
	$$
	where $K(\xi_1, \dots, \xi_n)$ is the curve obtained at the $n$th iteration in the construction of the $K(\xi)$. This means that
	$$
\frac 13 \mu(\epsilon_n) \leq 4 M_n \epsilon_n^2.
$$
The result follows.
\end{proof}

Notice that,
$$
\log L_n = n \sum_{a \in A} p_a(n) \ell(a) \quad \text{and} \quad \log M_n = n \sum_{a \in A} p_a(n) m(a), 
$$
where $(p_a(n), a \in A)$ is the empirical distribution for the proportions of the elements of $A$, i.e.\
$$
p_a(n) = \frac 1n \sum_{i = 1}^n \bone_{\xi_i = a}.
$$
Therefore, the key to understanding the log asymptotics of the inner tubular neighbourhoods, and hence the fractal dimension of $\partial D$, is the convergence of the empirical distribution to some limiting probability measure $(p_a, a \in A)$. This observation can easily be used to produce examples such as the following one of snowflakes whose boundary does not have a Minkowski dimension.

\begin{example}\label{ex::noMinkowskiDim}
	Consider the sequence $(\xi_n, n \in \N)$ taking values in $\{1,2\}$ defined by
	$$
	\xi_n =  \begin{cases} 1, & n \in S,\\ 2,& n \in \N \setminus S, \end{cases} \quad \text{where} \quad S = \bigcup_{n = 1}^\infty \{2^{2n} +1, \dots, 2^{2n+1}\}.
	$$
	Then, for $a \in \{1, 2\}$, it is easily checked that
	$$
	\liminf_{n \to \infty} p_a(n) = \frac 13 \quad \text{and} \quad \limsup_{n \to \infty} p_a(n) = \frac 23.
	$$
	Therefore,
	$$
	\liminf_{n \to \infty} \frac{\log M_n}{\log L_n} = \liminf_{n \to \infty} \frac{n^{-1} \sum_{i \leq n} \log m_i}{n^{-1} \sum_{i \leq n} \log \ell_i} = \frac{ \frac 13 \log 4 + \frac 23 \log 7}{\frac 13 \log 3 + \frac 23 \log 5} \simeq 1.2225,
	$$
	while
	$$
	\limsup_{n \to \infty} \frac{\log M_n}{\log L_n} = \limsup_{n \to \infty} \frac{n^{-1} \sum_{i \leq n} \log m_i}{n^{-1} \sum_{i \leq n} \log \ell_i} = \frac{ \frac 23 \log 4 + \frac 13 \log 7}{\frac 23 \log 3 + \frac 13 \log 5} \simeq 1.2395.
	$$
	So the Minkowski dimension of $K(\xi)$ does not exist.
\end{example}

On the other hand, when the sequence $(\xi_n, n \in \N)$ is stationary and ergodic, we have, for every $a \in A$,
$$
p_a(n) \to p_a,
$$
as $n \to \infty$, for some probability measure $(p_a, a \in A)$. In this case, we have the following elementary result.

\begin{theorem}\label{thm::dimensionScaleHomo}
	Suppose that $(\xi_n, n \in \N)$ is stationary and ergodic. Then, the Minkowski dimension of $\partial D$ exists, is equal to its Hausdorff dimension, and is given by
	\begin{equation*}\label{eq::gammaHausdorffMinkowski}
	\frac{\sum_{a \in A} p_a m(a)}{\sum_{a \in A} p_a \ell(a)}.
	\end{equation*}
\end{theorem}

In a fashion inspired by \cite{BH1997}, let us now focus on the stationary and ergodic case and study how the speed of convergence in the ergodic theorem affects the geometry of the snowflake. Assume that $p_a(n) \to p_a$ as $n \to \infty$ for some probability measure $(p_a, a \in A)$ and that
\begin{equation}\label{eq::definitionFunG}
\sup_{a \in A} |p_a(n) - p_a| \leq n^{-1} g(n),
\end{equation}
where $g$ is a regularly varying function. If $p_a \in (0,1)$ for every $a \in A$, then
$$
\liminf_{n \to \infty} |n p_a(n) - n p_a| > 0.
$$
Therefore, the best rate of convergence that one can get in general is $g(n) = O(1)$. As such, we will always assume that $g$ is non-decreasing.

We now show how the rate of convergence can be used to control the volume of the inner tubular neighbourhoods.

\begin{theorem}\label{thm::boundsVolumeScaleHomo}
Suppose that $(\xi_n, n \in \N)$ is stationary and ergodic and satisfies \eqref{eq::definitionFunG} for some regularly varying function $g$. Then,
$$
c_1 \epsilon^{2- \gamma} e^{-c_2 g(c_3 \log (1/\epsilon)) } \leq \mu(\epsilon)\leq c_4 \epsilon^{2- \gamma} e^{c_2 g(c_3 \log (1/\epsilon))}
$$
for small $\epsilon$, where $\gamma = \dim_M \partial D$.
\end{theorem}

\begin{proof}
	By Lemma \ref{lem::tubNeighbourhoodScaleHomo}, we know that
	$$
	\mu(\epsilon_n) \asymp \epsilon_n^{2- \gamma}M_n L_n^{-\gamma}.
	$$
	But, using Theorem \ref{thm::dimensionScaleHomo}, note that
	\begin{align*}
	\log (M_n L_n^{-\gamma}) & = n \sum_{a \in A} [\log m(a) - \gamma \log \ell(a)]p_a(n)\\
		& = n \sum_{a \in A} [\log m(a) - \gamma \log \ell(a)][p_a(n)- p_a] \\
		& = O(g(n)).
	\end{align*}
	
	Putting these observations together shows that
	$$
	c_1 \epsilon_n^{2- \gamma} e^{-c_2 g(n)} \leq \mu(\epsilon_n) \leq c_4 \epsilon_n^{2- \gamma} e^{c_2 g(n)}.
	$$
	The result follows after using that $g$ is non-increasing, and that $\log(1/\epsilon_n) \asymp n$ and $\epsilon_n \asymp \epsilon_{n+1}$, because $A$ is bounded.
\end{proof}

One of the consequences of this result is that we can only have
$$
0 < \cM_*(\partial D) \leq \cM^*(\partial D) < \infty,
$$
if the rate of convergence is the best possible, i.e.\ $g(n) = O(1)$. When this is not the case, but the sharpest function $g$ is known, one can ask about the fluctuations of the volume of the inner tubular neighbourhoods between the bounds given above. A central example where this can be done is when the sequence $(\xi_n, n \in \N)$ is i.i.d., in which case the rate of convergence
$$
g(n) = \sqrt{n \log \log n}
$$
is given by the law of the iterated logarithm; we discuss this case in the following theorem.

\begin{theorem}
Suppose that $(\xi_n, n \in \N)$ is i.i.d. Then,
$$
c_1 \epsilon^{2- \gamma} e^{-c_2 \psi(1/\epsilon)} \leq \mu(\epsilon) \leq c_3 \epsilon^{2- \gamma } e^{c_2 \psi(1/\epsilon)}
$$
for small $\epsilon$, where $\gamma = \dim_M \partial D$ and
$$
\psi(x) = \sqrt{\log x \log \log \log x}.
$$
Furthermore,
$$
\liminf_{\epsilon \to 0} \frac{\mu(\epsilon) e^{c_4 \psi(1/\epsilon)}}{\epsilon^{2- \gamma}} < \infty \quad \text{and} \quad \limsup_{\epsilon \to 0} \frac{\mu(\epsilon) e^{-c_4 \psi(1/\epsilon)}}{\epsilon^{2- \gamma}} > 0.
$$
\end{theorem}

\begin{proof}
The first part of the result follows from Theorem \ref{thm::boundsVolumeScaleHomo} using the form of $g$ for the i.i.d.\ case.

For the second part, notice that, by the law of the iterated logarithm,
$$
\log (M_n L_n^{-\gamma}) \leq - \frac 12 g(n), \text{ i.o.} \quad \text{and} \quad \log (M_n L_n^{-\gamma}) \geq \frac 12 g(n), \text{ i.o.}
$$
Proceeding as in the proof of Theorem \ref{thm::boundsVolumeScaleHomo} along the appropriate subsequences yields the result.
\end{proof}


\subsection{Heat content asymptotics}

We now study the heat content asymptotics of the scale homogeneous snowflakes. We rely on the results of Section \ref{sec::heatContentEstimates} which are readily applicable since the snowflakes are simply connected.

Let us start by finishing the discussion of Example \ref{ex::noMinkowskiDim}, a case where the heat content has non-trivial log asymptotics reflecting that the Minkowski dimension of the boundary does not exist.

\begin{example} For the snowflake constructed in Example \ref{ex::noMinkowskiDim}, we get, by Theorem \ref{thm::lowerAndUpperHeatDimension}, that
	$$
	\liminf_{s \to 0} \left(\frac d2 - \frac{\log E(s)}{\log s} \right) \simeq 0.6112,
	$$
	while
	$$
	\limsup_{s \to 0} \left(\frac d2 - \frac{\log E(s)}{\log s} \right) \simeq 0.6198.
	$$
	In other words, the log asymptotics of the heat content oscillate between those dictated by the lower and upper Minkowski dimensions.
\end{example}

As we did when we studied the volume of the inner tubular neighbourhood above, we now focus on the case where the sequence $(\xi_n, n \in \N)$ is stationary and ergodic with a rate of convergence given by the regularly varying function $g$ in \eqref{eq::definitionFunG}.

Using Theorems \ref{thm::vdBLowerBound}, \ref{thm::upperBoundHeatContent} or \ref{thm::heatContentAbelianArgument}, and \ref{thm::boundsVolumeScaleHomo}, it is straightforward to get lower and upper bounds for the heat content. Furthermore, when the volume of the inner tubular neighbourhoods oscillates between the lower and upper bounds given by $g$, then a similar reasoning along appropriate subsequences shows that the heat content oscillates in a similar fashion.

Notice, however, that dealing with the upper bound for the heat content is more delicate. This is because the function $\omega$ in Theorem \ref{thm::upperBoundHeatContent} typically involves a logarithmic correction (see the proof of Theorem \ref{thm::lowerAndUpperHeatDimension}). As we now show, this is never a problem when $g$ is not slowly varying; we postpone the discussion of the other situation to the next subsection.

\begin{theorem}\label{thm::heatContentFluctuationRegularlyVarying}
	Suppose that $(\xi_n, n \in \N)$ is stationary and ergodic and that $g$ is regularly varying, but not slowly varying, i.e.\ has the form
	$$
	g(x) = x^\theta L(x),
	$$
	where $\theta \in(0, \infty)$ and $L$ is slowly varying. Then,
	\begin{equation}\label{eq::boundHeatContent}
	c_1 s^{1- \gamma/2} e^{-c_2 g(c_3 \log(1/s))} \leq E(s) \leq c_4 s^{1- \gamma/2} e^{c_2 g(c_3 \log(1/s))}
	\end{equation}
	for small $s$, where $\gamma = \dim_M \partial D$. Furthermore, if
	\begin{equation}\label{eq::fluctuationsVolumeTubNeighbourhood}
	\liminf_{\epsilon \to 0} \frac{\mu(\epsilon) e^{c_5 g(c_6 \log(1/\epsilon))}}{\epsilon^{2- \gamma}} < \infty \quad \text{and} \quad \limsup_{\epsilon \to 0} \frac{\mu(\epsilon) e^{-c_5 g(c_6 \log(1/\epsilon))}}{\epsilon^{2- \gamma}}> 0,
	\end{equation}
	then
	\begin{equation}\label{eq::fluctuationsHeatContent}
	\liminf_{s \to 0} \frac{E(s) e^{c_7 g(c_8 \log(1/s))}}{s^{1- \gamma/2}} < \infty  \quad \text{and} \quad \limsup_{s \to 0}  \frac{E(s) e^{-c_7 g(c_8 \log(1/s))}}{s^{1- \gamma/2}}> 0.
	\end{equation}
\end{theorem}

\begin{proof}
	The lower bound in \eqref{eq::boundHeatContent} follows from Theorems \ref{thm::vdBLowerBound} and \ref{thm::boundsVolumeScaleHomo}. The limsup part of \eqref{eq::fluctuationsVolumeTubNeighbourhood} implies that
	$$
	\mu(\epsilon) \geq c_9 \epsilon^{2- \gamma} e^{c_{5} g(c_{6} \log (1/\epsilon))}, \text{ i.o.}
	$$
	Together with Theorem \ref{thm::vdBLowerBound}, this shows that
	$$
	E(s) \geq c_{10} s^{1- \gamma/2} e^{c_7 g(c_8 \log(1/s))}, \text{ i.o.}
	$$
	from which the limsup part of \eqref{eq::fluctuationsHeatContent} follows.
	
	To get the upper bound in \eqref{eq::boundHeatContent}, we rely on Theorem \ref{thm::upperBoundHeatContent} with the choice
	\begin{equation}\label{eq::defnOmega}
	\omega(s) = \sqrt{4 s \log (1/s)}
	\end{equation}
	and Theorem \ref{thm::boundsVolumeScaleHomo} to get that
	\begin{equation}\label{eq::technicalEqOne}
	\begin{aligned}
	E(s) & \leq \mu(\omega(s)) + 2^{(d+2)/2} \vol_d (D) e^{-\omega(s)^2/4s}\\
		& \leq c_{11} s^{1- \gamma/2} \left( \log \frac 1s \right)^{1- \gamma/2} e^{c_{12} g \left(c_{13} \log \frac{1}{s \log( 1/s)} \right)} + 2^{(d+2)/2} \vol_d (D) s\\
		& \leq c_{11} s^{1- \gamma/2} \left( \log \frac 1s \right)^{1- \gamma/2} e^{c_{12} g (c_3 \log (1/s))},
	\end{aligned}
	\end{equation}
	for small $s$. Now using that that $g(x) = x^\theta L(x)$ (this is where we use the assumption that $g$ is not slowly varying), we get that
	\begin{equation}\label{eq::technicalEqTwo}
	E(s) \leq c_4 s^{1- \gamma/2} e^{c_2 g (c_3 \log (1/s))},
	\end{equation}
	for small $s$, as required.
	
	Finally, the liminf part of \eqref{eq::fluctuationsVolumeTubNeighbourhood} implies that
	$$
	\mu(\epsilon) \leq c_{13} \epsilon^{2- \gamma} e^{-c_{14} g(c_{15} \log (1/\epsilon))}, \text{ i.o.}
	$$
	Reasoning as above then shows that
	$$
	E(s) \leq c_{19} s^{1- \gamma/2} e^{-c_7 g(c_8 \log(1/s))}, \text{ i.o.}
	$$
	This establishes the liminf part of \eqref{eq::fluctuationsHeatContent} and completes the proof.
\end{proof}

A particular instance of this is when $(\xi_n, n \in \N)$ is i.i.d.\ and $g(x) = \sqrt{x \log \log x}$; we state it in the following corollary.

\begin{corollary}
	Suppose that $(\xi_n, n \in \N)$ is i.i.d. Then,
	$$
	c_1 s^{1- \gamma/2} e^{-c_2 \psi(1/s)} \leq E(s) \leq c_3 s^{1- \gamma/2} e^{ -c_2 \psi(1/s)}
	$$
	for small $s$, where $\gamma = \dim_M \partial D$ and
	$$
	\psi(x) = \sqrt{ \log x \log \log \log x}.
	$$
	Furthermore,
	$$
	\liminf_{s \to 0} \frac{E(s) e^{c_4 \psi(1/s)}}{s^{1- \gamma/2}} < \infty  \quad \text{and} \quad \limsup_{s \to 0}  \frac{E(s) e^{-c_4 \psi(1/s)}}{s^{1- \gamma/2}}> 0.
	$$
\end{corollary}


\subsection{Slowly varying rates of convergence in the ergodic theorem}

The key element in the proof of Theorem \ref{thm::heatContentFluctuationRegularlyVarying} is that the logarithmic correction introduced in \eqref{eq::defnOmega} is \emph{not} felt because $g$ is \emph{not} slowly varying. This is what enables us to go from \eqref{eq::technicalEqOne} to \eqref{eq::technicalEqTwo}, which is no longer possible if $g(x) = \log \log x$ or $g$ is constant, for example.

These difficulties as $g$ becomes `closer to a constant' are expected. For example, for the triadic Koch snowflake studied intensively in \cite{FLV1994, LP2006}, it is known that the volume of the inner tubular neighbourhood behaves like
$$
\mu(\epsilon) = p(\log \epsilon) \epsilon^{2- \gamma} + o(\epsilon^{2- \gamma}),
$$
as $\epsilon \to 0$, for some non-constant $\log 3$ periodic function $p$; this corresponds to $g$ constant. By a renewal argument, this is known to imply that
$$
E(s) = q(\log s) s^{1- \gamma/2} +o (s^{1- \gamma/2}),
$$
as $s \to 0$, for some $\log 9$ periodic function $q$. But, to the best of my knowledge, it is not known whether or not $q$ is periodic. In other words, it is not known whether $E(s)$ fluctuates between the upper and lower bounds used in Theorem \ref{thm::heatContentFluctuationRegularlyVarying} for this constant function $g$.

Similar problems have also been studied by Lapidus and coauthors for the eigenvalue counting function
$$
N(\lambda) = \#\{ \text{eigenvalues of }-\Delta/2 \leq \lambda\};
$$
see \cite{LvF2000, LP1993, LP1996} and references therein. For a family of open sets called \emph{fractal strings}, the eigenvalue counting function satisfies
$$
c_1\lambda^{\gamma/2} \leq N(\lambda) - \sqrt{2/\pi} \lambda^{1/2} \leq c_2 \lambda^{\gamma/2},
$$
where $\gamma$ is the Minkowski dimension of the boundary of the fractal string, provided
$$
0 < \cM_* (\partial U) \leq \cM^*(\partial U) < \infty.
$$
This again corresponds to a setup where $g$ is constant; see the remark after Theorem \ref{thm::boundsVolumeScaleHomo}. As it turns out, the question of whether $N(\lambda) - \sqrt{2/\pi} \lambda^{d/2}$ fluctuates between its lower and upper bounds is equivalent to the Riemann hypothesis; see \cite{LvF2000} for further information.

Our aim in this subsection it to discuss one possible refinement of Theorem \ref{thm::upperBoundHeatContent} which enables us to extend Theorem \ref{thm::heatContentFluctuationRegularlyVarying} to some some cases where $g$ is slowly varying.
\begin{lemma}
	Let $D$ be a bounded domain of $\R^d$. Suppose that
	$$
	\mu(\epsilon) \leq c_1 \epsilon^{d- \gamma} e^{\pm c_2 g(c_3 \log(1/\epsilon))}
	$$
	for small $\epsilon$, for some slowly varying function $g$. Suppose further that, for some $k \in \N$,
	$$
	\log^k(x) = O(g(x)),
	$$
	as $x \to \infty$. Then,
	$$
	E(s) \leq c_4 s^{(d- \gamma)/2} e^{\pm c_5 g(c_6 \log(1/s))}
	$$
	for small $s$.
\end{lemma}

This somewhat technical result acts as a substitute for the steps in \eqref{eq::technicalEqOne} and \eqref{eq::technicalEqTwo} in the proof of Theorem \ref{thm::heatContentFluctuationRegularlyVarying}. Therefore, it readily extends Theorem \ref{thm::heatContentFluctuationRegularlyVarying} to situations where $g$ is slowly varying but satisfies the conditions of the lemma.

\begin{proof}
	For $i \in \{1, \dots, k + 1\}$, define
	$$
	\omega_i(s) = \sqrt{2 d s \log^i (1/s)},
	$$
	where $\log^i = \log \circ \cdots \circ \log$, with $i-1$ compositions, and put
	\begin{align*}
	B^{k+1}_s &= \{x \in D : d(x, \partial D) \leq \omega_{k+1}(s) \},\\
	B_s^i &= \{ x \in D: d(x, \partial D) \leq \omega_i(s)\} \setminus B_s^{i+1}, \quad i \in\{1, \dots, k-1\},\\
	B^0_s &= D \setminus \bigcup_{i = 1}^{k+1} B_s^i.
	\end{align*}
	
	By \eqref{eq::probabilisticSolution} and \eqref{eq::gaussianEstimate}, we see that
	\begin{align*}
	E(s)&\leq \sum_{i = 0}^{k+1} \int_{B_s^i} \bP_x(T_{D^c} \leq s)	dx\\
		& \leq 2^{(d+2)/2} \vol_d(D) s^{d/2} + 2^{(d+2)/2} \sum_{i=1}^{k} \mu(\omega_{i}(s)) e^{- \omega_{i+1}(s)^2/4s} + \mu(\omega_{k+1}(s)).
	\end{align*}
	By our assumption on $g$, we have
	\begin{align*}
	\mu(\omega_{k+1}(s)) & \leq c_1 s^{(d- \gamma)/2} \left( \log^{k+1} \frac 1s\right)^{(d- \gamma)/2} e^{\pm c_2 g(c_6 \log(1/s))}\\
		& \leq c_1 s^{(d- \gamma)/2} e^{\pm c_7 g(c_6 \log(1/s))}
	\end{align*}
	for $s$ small. Furthermore, for $i \in \{1, \dots, k\}$,
	\begin{align*}
		\mu(\omega_i(s)) e^{-\omega_{i+1}(s)^2/4s} & \leq c_1 s^{(d- \gamma)/2} \left(\log^i \frac 1s \right)^{(d- \gamma)/2} e^{\pm c_2 g(c_6 \log (1/s))} e^{- \frac d2 \log^{i+1}(1/s)}\\
		& = c_1 s^{(d- \gamma)/2} \left( \log^i \frac 1s \right)^{- \gamma/2} e^{\pm c_2 g(c_6 \log(1/s))}\\
		&\leq c_8 s^{(d- \gamma)/2}  e^{\pm c_2 g(c_6 \log(1/s))}
	\end{align*}
	for $s$ small.
	
	Combining these estimates gives the desired bound on the heat content.
\end{proof}


\section{General branching processes}

\label{sec::gbp}

In this section, we provide a brief introduction to general branching processes and introduce the relevant notation. The presentation is inspired by \cite{Hambly2000, Jagers1975, Nerman1981}, where the reader is referred for further information.


\subsection{Definitions and elementary properties}

The typical individual $x$ in a general branching process has offspring whose birth times are modelled by a point process $\xi_x$ on $(0, \infty)$, a lifetime modelled as a random variable $L_x$, and a \emph{characteristic} which is a (possibly random) c\`adl\`ag function $\phi$ on $\R$. The triples $(\xi_x, L_x, \phi_x)_x$ are sometimes assumed to be i.i.d.\ but we will allow $\phi_x$ to depend on the progeny of $x$; also, we do \emph{not} make any assumptions about the joint distribution of $(\xi_x, L_x, \phi_x)$. When discussing a generic individual, it is convenient to drop the dependence on $x$ and write $(\xi, L, \phi)$. We shall use the notation
$$
\xi(t) = \xi((0, t]), \quad \nu(dt) = \bE \xi(dt), \quad \xi_\gamma(dt) = e^{-\gamma t} \xi(dt) \quad \text{and} \quad \nu_\gamma(dt) = \bE \xi_\gamma(dt).
$$

We assume that the process has a \emph{Malthusian parameter} $\gamma \in (0, \infty)$ for which
\begin{equation}\label{eq::MalthusianParameter}
\nu_\gamma(\infty) = 1.
\end{equation}
In particular, the general branching process is super-critical, i.e.\ $\nu(\infty) > 1$. We also assume that $\nu_\gamma$ has a finite first moment.

It is natural to index the individuals of the population by their ancestry, which is the random subtree $\cT$ of the set of finite words
\begin{equation}\label{eq::addressSpace}
I = \bigcup_{k = 0}^\infty \N^{k}, \quad \text{with} \quad \N^0 = \emptyset,
\end{equation}
generated by the underlying Galton-Watson process. The birth time of $x$ is written $\sigma_x$ and we have the relation
$$
\xi_x = \sum_{i = 1}^{\xi_x(\infty)} \delta_{\sigma_{x,i} - \sigma_x},
$$
where $\delta$ is the Dirac measure and $x,i$ is the concatenation of the words $x$ and $i$.

The individuals of the population are counted using the characteristic $\phi$ through the \emph{characteristic counting process} $Z^\phi$ defined by
$$
Z^\phi(t) = \sum_{x \in \cT} \phi_x (t - \sigma_x) = \phi_\emptyset(t) + \sum_{i = 1}^{\xi_\emptyset(\infty)} Z_i^\phi(t- \sigma_i),
$$
where the $Z^\phi_i$ are i.i.d.\ copies of $Z^\phi$. Later, we will define characteristic functions whose corresponding counting process contains information about the inner Minkowski content and the heat content.

Results are often proved under the assumption that $\phi$ vanishes for negative times; e.g.\ \cite{Nerman1981}. When this assumption is not satisfied, we may work with $Z^\phi \bone_{[0, \infty)}$ instead of $Z^\phi$ to study the asymptotics of $Z^\phi$ as $t \to \infty$. Indeed, $Z^\phi \bone_{[0, \infty)}$ also appears as a counting process since
\begin{equation}\label{eq::oneSidedToTwoSided}
\begin{aligned}
	Z^\chi(t) & = Z^\phi(t) \bone_{t \geq 0}\\
			& = \phi_\emptyset(t) \bone_{t \geq 0} + \sum_{i = 1}^{\xi_\emptyset(\infty)} Z_i^\phi(t- \sigma_i) \bone_{0 \leq t < \sigma_i} + \sum_{i=1}^{\xi_\emptyset(\infty)} Z_i^\phi(t- \sigma_i) \bone_{t- \sigma_i \geq 0}\\
			& = \chi_\emptyset(t) + \sum_{i=1}^{\xi_\emptyset(\infty)} Z_i^\chi(t- \sigma_i),
\end{aligned}
\end{equation}
where the first, respectively last, equation defines $Z^\chi$, respectively $\chi$. It is clear that $\chi$ is a characteristic (that depends on the progeny in general), that $Z^\chi$ is its corresponding counting process, and that the asymptotics of $Z^\phi$ and $Z^\chi$ as $t \to \infty$ are the same.

The growth of the population of a general branching process is captured by the process $M$ defined by
\begin{equation*}\label{eq::fundamentalMartingale}
M_t = \sum_{x \in \Lambda_t} e^{-\gamma \sigma_x}, \quad \text{where} \quad \Lambda_t = \{x = (y, i) : \sigma_{y} \leq t < \sigma_x\}
\end{equation*}
is the set of individuals born after time $t$ to parents born up to time $t$. The process $M$ is a non-negative c\`adl\`ag $\cF_t$-martingale with unit expectation, where
$$
\cF_t = \sigma((\xi_x, L_x) : \sigma_x \leq t).
$$
By martingale convergence, $M_t \to M_\infty$, almost surely, as $t \to \infty$. Furthermore, if
\begin{equation}\label{eq::xlogxCondition}
	\bE\left[\xi_\gamma(\infty) (\log \xi_\gamma(\infty))_+ \right] < \infty,
\end{equation}
then $M$ is uniformly integrable and $M_\infty$ is positive on the event that there is no extinction. Proofs of these facts may be found in \cite{Jagers1975, Nerman1981,Doney1972, Doney1976}.


\subsection{Strong law of large numbers} In this subsection, we state  the strong law of large numbers proved by Nerman in \cite{Nerman1981}. The result assumes the characteristic is non-negative; in applications, it suffices to write the characteristic as the difference of its positive and negative parts.

We will need the following regularity condition.

\begin{condition}\label{cond::ghFunCondition}
There exist non-increasing bounded positive integrable c\`adl\`ag functions $g$ and $h$ on $[0, \infty)$ such that
$$
\bE \left[ \sup_{t \geq 0} \frac{ \xi_\gamma(\infty) - \xi_\gamma(t)}{g(t)} \right] < \infty \quad \text{and} \quad \bE \left[ \sup_{t \geq 0} \frac{e^{-\gamma t} \phi(t)}{h(t)} \right] < \infty.
$$	
\end{condition}

The first part of the condition is satisfied if the expected number of offspring is finite because then, choosing $g(t) = 1 \wedge t^{-2}$, we have
$$
\frac{\xi_\gamma(\infty)- \xi_\gamma(t)}{g(t)} \leq \int_t^\infty \frac{1}{g(s)}\xi_\gamma(ds) \leq \int_0^\infty \frac{1}{g(s)} \xi_\gamma(ds) \leq \sup_{u \geq 0} \{(1 \wedge u^2) e^{-\gamma u}\} \xi(\infty),
$$
which has finite expectation.

We can now state the strong law of large numbers.

\begin{theorem}[Nerman]\label{thm::NermanSLLN}
	Let $(\xi_x, L_x, \phi_x)_x$ be a general branching process with Malthusian parameter $\gamma$, where $\phi \geq 0$ and $\phi(t) = 0$ for $t < 0$. Assume that $\nu_\gamma$ is non-lattice. Assume further that Condition \ref{cond::ghFunCondition} is satisfied.
	
	Then,
	$$
	z^\phi(t) \to z^\phi(\infty) = \frac{\int_0^\infty e^{- \gamma s} \bE \phi(s) ds}{\int_0^\infty s \nu_\gamma(ds)},
	$$
	some finite constant, and
	$$
	e^{-\gamma t} Z^\phi(t) \to z^\phi(\infty) M_\infty, \text{ a.s.},
	$$
	as $t \to \infty$, where $M_\infty$ is the almost sure limit of the fundamental martingale of the general branching process. Furthermore, if $M$ is uniformly integrable, then the convergence also takes place in $L^1$.
\end{theorem}

A similar result holds when $\nu_\gamma$ is lattice; see \cite{Gatzouras2000}. But we will carefully avoid this case here.


\subsection{Applications to statistically self-similar fractals}
\label{subsec::appGBPtoFractals}

Intuitively, a random compact subset $K$ of $\R^d$ is statistically self-similar if there is a random number $N$ and random contracting similitudes $\Phi_1, \dots, \Phi_N$ such that
$$
K = \bigcup_{i=1}^N \Phi_i(K_i), \text { a.s.},
$$
where $K_1, \dots, K_N$ are i.i.d.\ copies of $K$. The reader is referred to \cite{Falconer1986, Graf1987, MW1986} for more information.

To encode $K$ as general branching process, we use the address space $I$ defined in \eqref{eq::addressSpace}. To each $x \in I$, we associate a random collection $(N_x, \Phi_{x,1}, \dots \Phi_{x,N_x})_{x \in I}$, where $N_x$ is a natural number and $\Phi_{x,i}$ are contracting similitudes whose ratios we write $R_{x,i}$. We assume that the collection is i.i.d.\ in $x$.

Write $\cT$ for the path of the Galton-Watson process generated by the random numbers $(N_x, x \in I)$, i.e.\ $\emptyset \in \cT$ and
$$
\quad y = y_1 \dots y_n \in \cT \iff y_1 \dots y_{n-1} \in \cT \text{ and } y_n \leq N_{y_1\dots y_{n-1}}.
$$
Starting with a compact set $K_\emptyset$, define, for $x = x_1 \dots x_n \in \cT$,
$$
K_x = \Phi_{x_1} \circ \dots \circ \Phi_{x_1 \dots x_n}(K_\emptyset) \quad \text{and} \quad K = \bigcap_{n = 1}^\infty \bigcup_{|x| = n} K_x,
$$
where $|x|$ is the length of the word $x$. Then $K$ is the statistically self-similar set corresponding to $K_\emptyset$ and $(N, \Phi_1, \dots, \Phi_N)$.\footnote{To be completely rigorous, we assume that the sets $(\interior K_x, x \in \cT)$ \emph{form a net}, i.e.\
$$
x \leq y \implies \interior K_y \subset \interior K_x \quad \text{and} \quad \interior K_x \cap \interior K_y = \emptyset \text{ if neither $x \leq y$ nor $y \leq x$}.
$$

We also assume that the construction described above is \emph{proper} in the sense of \cite{Falconer1986a}, i.e.\ that every cut of $\cC$ of $\cT$ satisfies the condition: For every $x \in \cC$, there exists a point in $K_x$ that does not lie in any other $K_y$ with $y \in \cC$.}

The Hausdorff dimension of statistically self-similar sets is given by the following formula \emph{\`a la} Moran and Hutchinson \cite{Moran1946,Hutchinson1981}. The statement is adapted from \cite{Falconer1986, Graf1987, MW1986}.

\begin{theorem}\label{thm::dimHausStatSelfSimilar}
	Let $K$ be a random statistically self-similar set as above. Then, on the event that the set $K$ is not empty,
	$$
	\dim K = \inf\left\{ s : \bE \left(\sum_{i=1}^N R_i^s \right) \leq 1 \right\}, \text{ a.s.}
	$$
\end{theorem}

To make the underlying Galton-Watson process structure of statistically self-similar sets into a general branching process, we specify birth times by setting
$$
\xi_x = \sum_{i=1}^{N_x} \delta_{-\log R_{x,i}},
$$
and the lifetimes by setting $L_x = \sup_i (\sigma_{x,i} - \sigma_x)$.

For the first generation of offspring $e^{-\sigma_i} = R_i$. More generally, with this parametrisation, the offspring $x$ born around time $t$ correspond to compact sets $K_x$ of size roughly $e^{-t}$ in the construction.

Notice that
\begin{equation}\label{eq::calculationHausdorffDim}
\bE\int_{0}^\infty e^{-s x} \xi(dx) = \bE \left( \sum_{i=1}^N R_i^s \right),
\end{equation}
so that the Malthusian parameter of the underlying general branching process is equal to the almost sure Hausdorff dimension of the set $K$; compare \eqref{eq::calculationHausdorffDim} with \eqref{eq::MalthusianParameter}.


\section{Statistically self-similar snowflakes}
\label{sec::selfSimilar}

In this section, we study the geometry and heat content asymptotics of a family of Koch type snowflakes whose boundary is statistically self-similar. Our main tool will be Nerman's strong law of large numbers.

Let us stress that the discussion here can be extended to any other snowflakes with statistically self-similar boundary so long as the corresponding general branching process satisfies the assumptions of the strong law of large numbers. However, for brevity, we will focus on a particular class of snowflakes that can easily be connected to those discussed in Section \ref{sec::scaleHomogeneous}. Our aim is to emphasise that the behaviour of the heat content is qualitatively different when the boundary of the snowflake is statistically self-similar and not space homogeneous, because of additional spatial independence.


\subsection{Construction}

Start with the segment $[0,1]$ and pick $a \in A$, a bounded subset of $\N$, randomly according to the probability distribution $(p_a, a \in A)$. Consider the building block $K(a)$ described in Section \ref{sec::scaleHomogeneous}. Then replace each linear piece of $K(a)$ by a scaled i.i.d.\ copy of itself, i.e.\ again using $(p_a, a \in A)$. Iterating indefinitely, we obtain a sequence of curves converging to a statistically self-similar curve $K$. Indeed, in the notation of Subsection \ref{subsec::appGBPtoFractals}, it suffices to set
$$
(N(a), R_1(a), \dots, R_N(a)) = (m(a), \ell(a)^{-1}, \dots, \ell(a)^{-1});
$$
the maps $(\Phi_1, \dots, \Phi_N)$ can easily be deduced from this.

The corresponding general branching process $(\xi,L)$ (no characteristic just yet) is obtained as detailed above. Figure \ref{fig::selfSimilarCurve} contains an approximation of $K$ when $A = \{1, 2,3 \}$ and $(p_a, a \in A)$ is uniform.

\begin{figure}
\includegraphics[width = 10 cm]{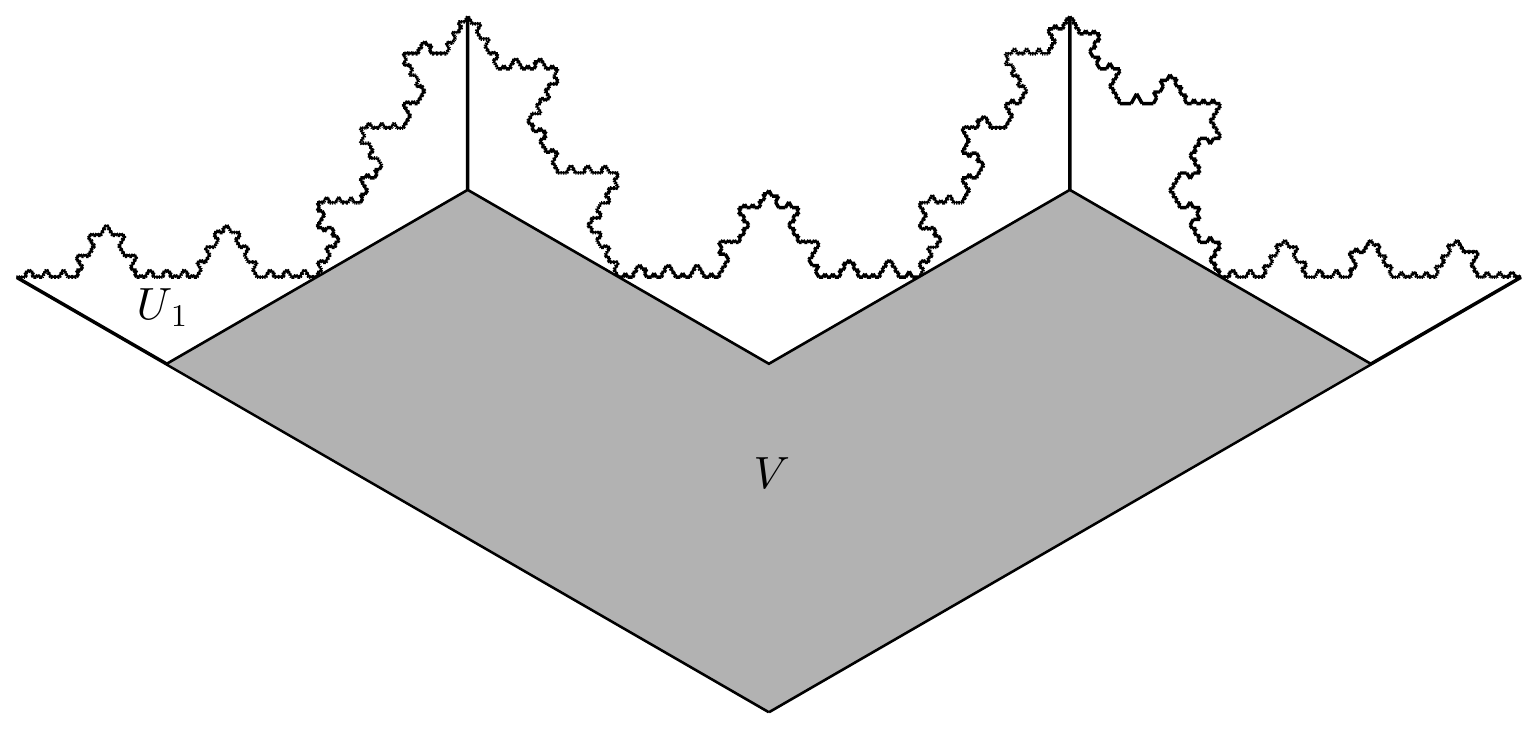}
\caption{Statistically self-similar curve $K$ for a third of the snowflake $U$ with its natural decomposition.}
\label{fig::selfSimilarCurve}
\end{figure}

Finally, the snowflake $D$ is defined as the simply connected interior of the Jordan curve created using three i.i.d.\ copies of $K$.

To ensure that $\nu_\gamma$ is non-lattice, we will make the following assumption.

\begin{assumption}\label{ass::nonLattice}
The set $A$ contains two elements $a_1$ and $a_2$ such that
$$
p_{a_1} \text{ and } p_{a_2} \in (0,1) \quad \text{and} \quad \frac{\log \ell(a_1)}{\log \ell(a_2)} \notin \Q.
$$	
\end{assumption}


\subsection{Fractal dimension and Minkowski content}

Now that we have defined the statistically self-similar snowflakes, we will use the theory of general branching processes to study the volume of their inner tubular neighbourhoods.

\begin{theorem}\label{thm::minkowskiContentSnowflakeSLLN}
	For the statistically self-similar snowflake $D$, we have
	$$
	\epsilon^{\gamma- 2} \mu(\epsilon) \to \cM N_\infty, \text{ a.s.} \text{ and in } L^1,
	$$
	as $\epsilon \to 0$, for some positive constant $\cM$ and positive random variable $N_\infty$ with unit expectation.
	
	In particular, the Minkowski dimension of $\partial D$ exists and is equal to $\gamma$ almost surely, and the inner Minkowski content exists, is finite, and given by $\cM N_\infty$ almost surely.
\end{theorem}

Notice that the reasoning used in Lemma \ref{lem::tubNeighbourhoodScaleHomo} shows that in this situation, again, the inner Minkowski dimension coincides with the standard definition.

\begin{proof}
	The snowflake $D$ is built from three i.i.d.\ copies of $U$, the open set whose boundary is made of two linear pieces and $K$, as depicted in Figure \ref{fig::selfSimilarCurve}.

	We will focus on how to deal with one such third $U$ for now. To do that, put
	$$
	\tilde \mu_U(\epsilon) = \vol_2(\{x \in U : d(x, K) \leq \epsilon\}),
	$$
	i.e.\ $\tilde \mu$ measures the volume of the part of the inner tubular neighbourhood close to the fractal part $K$ of the boundary of $U$ only.
	
	By construction of $K$, the set $U$ is made of a polygonal region $V$ (shaded in Figure in Figure \ref{fig::selfSimilarCurve}), say, and $m(a)$ copies $U_i$ of itself scaled by a random factor $R_i$. Therefore,
	\begin{equation}\label{eq::decompostitionVolumeTube}
	\tilde \mu_U(\epsilon) = \theta(\epsilon) + \sum_{i=1}^{m(a)} \tilde \mu_{R_i U_i}(\epsilon),
	\end{equation}
	where $\theta(\epsilon)$ is an error term bounded by $c_1 \epsilon^2$, as is clear from Figure \ref{fig::selfSimilarCurve} (the existence of the constant $c_1$ also uses that the set $A$ and therefore $m(a)$ is bounded). Furthermore, by scaling,
	\begin{equation}\label{eq::volumeScaling}
	\tilde \mu_{R_i U_i}(\epsilon) = R_i^2 \tilde \mu_{U_i}(R_i^{-1} \epsilon).
	\end{equation}
	
	In the language of the general branching process, each $R_i U_i$ corresponds to an offspring of the fractal $K$ born at time $\sigma_i = - \log R_i$. Therefore, putting, for $t \in \R$,
	$$
	Z^\phi(t) = e^{2t} \tilde \mu_U (e^{-t}) \quad \text{and} \quad \phi(t) = e^{2t} \theta(e^{-t}),
	$$
	the relations \eqref{eq::decompostitionVolumeTube} and \eqref{eq::volumeScaling} combine to produce
	$$
	Z^\phi(t) = \phi(t) + \sum_{i = 1}^{\xi(\infty)} Z_i^\phi(t- \sigma_i),
	$$
	where the $Z_i^\phi$ are i.i.d.\ copies of $Z$ and $\phi$ is bounded. Therefore, $Z^\phi$ is the counting process of the characteristic $\phi$.
	
	To apply the strong law of large numbers, we need to consider
	$$
	Z^\chi(t) = Z^\phi(t) \bone_{t \geq 0} = \chi(t) + \sum_{i = 1}^{\xi(\infty)} Z_i^\chi(t- \sigma_i),
	$$
	defined using \eqref{eq::oneSidedToTwoSided}.
	
	Since $\xi(\infty)$ is bounded (because $A$ is) and $Z^\phi(t)$ is bounded for negative times (by $\vol_2(U)$), the characteristic $\chi$ must be bounded as well. These observations imply that Condition \ref{cond::ghFunCondition} is satisfied as well as the integrability condition of \eqref{eq::xlogxCondition}. Therefore, by Theorem \ref{thm::NermanSLLN},
		$$
		e^{-\gamma t} Z^\chi(t) \to z^\chi(\infty) M_\infty \text{ a.s.} \text{ and in } L^1,
		$$
		as $t \to \infty$, where
		$$
		z^\chi(\infty) = \frac{\int_0^\infty u^\chi(t) dt}{\int_0^\infty t \nu_\gamma(dt)} \in (0, \infty).
		$$
		By definition of $Z^\chi$, this shows that
		\begin{equation}\label{eq::muepsilonconvergencethird}
			\epsilon^{\gamma - 2} \tilde \mu_U (\epsilon) \to z^\chi(\infty) M_\infty.
		\end{equation}
		
		As $\mu$ is the sum of three i.i.d.\ copies of $\tilde \mu_U$, we get the desired result by putting $\cM = 3 z^\chi(\infty)$ and setting $N_\infty$ to be a third of the sum of the three i.i.d.\ copies of $M_\infty$ given by the general branching process.
	\end{proof}


\subsection{Heat content asymptotics}

Following the analysis performed above for the volume of the inner tubular neighbourhoods of $D$, we now use the theory of general branching processes to study the heat content of $D$. Our aim is to prove the following result.

\begin{theorem}\label{thm::heatContentAsymptotics}
	For the statistically self-similar Koch snowflake $D$
	$$
	s^{\gamma/2-1} E(s) \to \cE N_\infty, \text{a.s.} \text{ and in } L^1,
	$$
	as $s \to 0$, for some positive constant $\cE$ and positive random variable $N_\infty$ with unit expectation, where $\gamma = \dim_M \partial D$.
\end{theorem}

The random variable appearing in this theorem is the same as that appearing in Theorem \ref{thm::minkowskiContentSnowflakeSLLN}. So this theorem implies that, for statistically self-similar snowflakes, one can recover both the Minkowski dimension and the inner Minkowski content from short time asymptotics of the heat content.

In the proof, we will use the elementary fact that if $D$ is a domain in $\R^d$ and $r \in (0, \infty)$, then
\begin{equation}\label{eq::scalingHeatContent}
E_{rD}(s) = r^{d} E_D(r^{-2} s).
\end{equation}

\begin{proof}
	As in the proof for the volume of the inner tubular neighbourhood, we first only consider a third $U$ of the snowflake and let $w$ be the solution of the heat equation with the boundary condition
		$$
		w(s, x) = \begin{cases} 1, & x \text{ in the fractal part of } \partial U,\\
		0, & x \text{ is in the linear part of } \partial U, \end{cases}
		$$
		and the initial condition
		$$
		w(0, x) = 0,\quad  x \in U.
		$$
	Now write $F_U$ for the heat content of $U$ with this altered boundary condition, i.e.\
	$$
	F_U(s) = \int_{U} w(s,x) dx.
	$$
	
	Before we can use the theory of general branching processes, we need some notation to understand the effect of adding extra cooling inside $U$. More precisely, write $\tilde w$ for the solution of the heat equation with boundary condition
	$$
	\tilde w(s, x) = \begin{cases} 1, & x \text{ in the fractal part of } \partial U,\\
	0, & x \text{ is in the linear part of } \partial U \text{ or in }\partial V,
	\end{cases}
	$$
	where $V$ is the polygonal region defined in the proof of Theorem \ref{thm::minkowskiContentSnowflakeSLLN}, and the initial condition
	$$
	\tilde w(0, x) = 0,\quad  x \in U.
	$$
	Let $\tilde F_U$ be the corresponding heat content
	$$
	\tilde F_U(s) = \int_{U} \tilde w(s,x) dx.
	$$
	
	With these definitions, we have
	$$
	F_U(s) = F_U(s) - \tilde F_U(s) + \sum_{i = 1}^{m(a)} F_{U_i} (s) = \psi(s) + \sum_{i = 1}^{m(a)} F_{U_i} (s),
	$$
	say. By scaling \eqref{eq::scalingHeatContent},
	$$
	F_{R_i U_i}(s) = R_i^2 F_{U_i}(R_i^{-2} s) = e^{-2\sigma_i} F_{U_i} (e^{2\sigma_i} s).
	$$
	Putting, for $t \in \R$,
	$$
	Z^\phi(t) = e^{2t} F_U(e^{-2t}) \quad \text{and} \quad \phi(t) = e^{2t}\psi(e^{-2t}),
	$$
	yields
	$$
	Z^\phi(t) = \phi(t) + \sum_{i = 1}^{\xi(\infty)} Z_i(t- \sigma_i),
	$$
	where the $Z_i^\phi$ are i.i.d.\ copies of $Z^\phi$.
	
	Let us now show that $\phi$ is bounded. This is done using the results of \cite{vdBG1997} about the impact on the heat content of imposing extra cooling. Put
	$$
	\lambda(\epsilon) = \vol_2 (\{ x \in \cup_{i} U_i : d(x,S) < \epsilon/\sqrt{2}  \text{ and } d(x, V) < \epsilon/\sqrt{2} \}),
	$$
	where $S$ is the fractal part of $\partial U$, i.e.\
	$$
	S = \overline{\{x \in \partial U : w(x, - ) = 1\}}.
	$$
	It is easy to check (by drawing a picture) that
	$$
	\lambda(\epsilon) \leq c_1 \epsilon^2.
	$$
	Therefore, by Corollary 1.3 of \cite{vdBG1997} and an integration by parts,
	$$
	0 \leq \psi(s)  \leq c_2 \int_0^\infty e^{-\epsilon^2/4s} \lambda(d \epsilon)
		 = c_2 2^{-1} s^{-1} \int_0^\infty \lambda(\epsilon) \epsilon e^{- \epsilon^2/4s} d \epsilon
		 \leq c_3 s \wedge c_4.
	$$
	From this, it follows that $\phi$ is bounded.
	
	Now reasoning as in the proof of Theorem \ref{thm::minkowskiContentSnowflakeSLLN} shows that
	$$
	s^{\gamma/2 -1} F_U(s) \to \cF M_\infty, \text{ a.s.\ and in } L^1, 
	$$
	as $s \to 0$ for some positive constant $\cF$.
	
	To conclude, recall that $D$ is the union of three i.i.d.\ copies of $U$, say $U_1, \dots, U_3$. Furthermore, by the estimate of \cite{vdBG1997} again, we have
	$$
	E_D(s) = \sum_{i= 1}^3 F_{U_i}(s) + O(s),
	$$
	as $s \to 0$. It follows that
	$$
	s^{\gamma/2 -1} E_D(s) \to \cE N_\infty, \text{ a.s.\ and in } L^1,
	$$
	as $s \to 0$, where $\cE = 3 \cF$ and $N_\infty$ is defined as in the proof of Theorem \ref{thm::minkowskiContentSnowflakeSLLN}, as required.
\end{proof}


\section{Open question}

For the statistically self-similar snowflakes, it is natural to ask about the fluctuations of the heat content around the almost sure short time asymptotics. 

In the forthcoming paper \cite{CCH2014}, we discuss a central limit theorem for general branching processes and apply it to study the fluctuations of the spectrum of some statistically self-similar fractals with Dirichlet weights.

More work on general branching processes is required before that central limit theorem can be applied to the snowflakes discussed here. However, it naturally leads to the following conjecture.

\begin{conjecture}
Let $D$ be a statistically self-similar snowflake of Section \ref{sec::selfSimilar}. Then,
$$
s^{- \gamma/4}(s^{\gamma/2 - 1 }E_D(s) - \cE N_\infty) \to_d Y_\infty,
$$
where $\cE$ and $N_\infty$ are defined in Section \ref{sec::selfSimilar} and $Y_\infty$ is a random variable whose characteristic function has the form
$$
\bE\left[e^{i \theta Y_\infty}\right] = \bE \left[e^{-\frac 12 \theta^2 \sigma^2 N_\infty}\right],
$$
for some $\sigma \in (0, \infty)$.
\end{conjecture}


\section{Acknowledgments}

I would like to thank Dmitry Belyaev, Ben Hambly, Sean Ledger and Sam Watson for related discussions, and Kolyan Ray for comments on an earlier version of this paper. The financial support of the Berrow Foundation and the Swiss National Science Foundation is gratefully acknowledged.


\bibliographystyle{alpha}
\bibliography{references}

\begin{thebibliography}{vdBdH99}

\bibitem[Bas95]{Bass1995}
R.F.\ Bass.
\newblock {\em Probabilistic techniques in analysis}.
\newblock Probability and its Applications (New York). Springer-Verlag, New
  York, 1995.

\bibitem[BC86]{BC1986}
J.\ Brossard and R.\ Carmona.
\newblock Can one hear the dimension of a fractal?
\newblock {\em Communications in Mathematical Physics}, 104(1):103--122, 1986.

\bibitem[BCDS94]{Buseretal1994}
P.\ Buser, J.\ Conway, P.\ Doyle, and K.-D.\ Semmler.
\newblock Some planar isospectral domains.
\newblock {\em International Mathematics Research Notices}, (9):391ff.,
  approx.\ 9 pp.\ (electronic), 1994.

\bibitem[Bed84]{Bedford1984}
T.\ Bedford.
\newblock Crinkly curves, {M}arkov partitions, and box dimensions in
  self-similar sets.
\newblock {\em Ph.D.\ Thesis, University of Warwick}, 1984.

\bibitem[Ber79]{Berry1979}
M.V.\ Berry.
\newblock Distribution of modes in fractal resonators.
\newblock In {\em Structural stability in physics ({P}roc. {I}nternat.
  {S}ymposia {A}ppl. {C}atastrophe {T}heory and {T}opological {C}oncepts in
  {P}hys., {I}nst. {I}nform. {S}ci., {U}niv. {T}\"ubingen, {T}\"ubingen,
  1978)}, volume~4 of {\em Springer Ser. Synergetics}, pages 51--53. Springer,
  Berlin, 1979.

\bibitem[Ber80]{Berry1980}
M.V.\ Berry.
\newblock Some geometric aspects of wave motion: wavefront dislocations,
  diffraction catastrophes, diffractals.
\newblock In {\em Geometry of the {L}aplace operator ({P}roc. {S}ympos. {P}ure
  {M}ath., {U}niv. {H}awaii, {H}onolulu, {H}awaii, 1979)}, Proc. Sympos. Pure
  Math., XXXVI, pages 13--28. Amer. Math. Soc., Providence, R.I., 1980.

\bibitem[BH97]{BH1997}
M.T.\ Barlow and B.~M.\ Hambly.
\newblock Transition density estimates for {B}rownian motion on scale irregular
  {S}ierpinski gaskets.
\newblock {\em Annales de l'Institut Henri Poincar\'e. Probabilit\'es et
  Statistiques}, 33(5):531--557, 1997.

\bibitem[CCH14]{CCH2014}
P.H.A.\ Charmoy, D.A.\ Croydon, and B.M.\ Hambly.
\newblock Central limit theorems for the spectra of random self-similar
  fractals with {D}irichlet weights.
\newblock {\em In preparation}, 2014.

\bibitem[Don72]{Doney1972}
R.A.\ Doney.
\newblock A limit theorem for a class of supercritical branching processes.
\newblock {\em Journal of Applied Probability}, 9:707--724, 1972.

\bibitem[Don76]{Doney1976}
R.A.\ Doney.
\newblock On single- and multi-type general age-dependent branching processes.
\newblock {\em Journal of Applied Probability}, 13(2):239--246, 1976.

\bibitem[Fal86a]{Falconer1986a}
K.J.\ Falconer.
\newblock {\em The geometry of fractal sets}, volume~85 of {\em Cambridge
  Tracts in Mathematics}.
\newblock Cambridge University Press, Cambridge, 1986.

\bibitem[Fal86b]{Falconer1986}
K.J.\ Falconer.
\newblock Random fractals.
\newblock {\em Mathematical Proceedings of the Cambridge Philosophical
  Society}, 6(3):559--583, 1986.

\bibitem[Fel68]{Feller1968}
W.\ Feller.
\newblock {\em An introduction to probability theory and its applications.
  {V}olumes {I} and {II}}.
\newblock Third edition. John Wiley \& Sons Inc., New York, 1968.

\bibitem[FLV95]{FLV1994}
J.\ Fleckinger, M.\ Levitin, and D.\ Vassiliev.
\newblock Heat equation on the triadic von {K}och snowflake: asymptotic and
  numerical analysis.
\newblock {\em Proceedings of the London Mathematical Society (3)},
  71(2):372--396, 1995.

\bibitem[Gat00]{Gatzouras2000}
D.\ Gatzouras.
\newblock On the lattice case of an almost-sure renewal theorem for branching
  random walks.
\newblock {\em Advances in Applied Probability}, 32(3):720--737, 2000.

\bibitem[Gra87]{Graf1987}
S.\ Graf.
\newblock Statistically self-similar fractals.
\newblock {\em Probabilty Theory and Related Fields}, 74(3):357--392, 1987.

\bibitem[GWW92]{Gordonetal1992}
C.\ Gordon, D.L.\ Webb, and S.\ Wolpert.
\newblock One cannot hear the shape of a drum.
\newblock {\em American Mathematical Society. Bulletin. New Series},
  27(1):134--138, 1992.

\bibitem[Ham92]{Hambly1992}
B.M.\ Hambly.
\newblock Brownian motion on a homogeneous random fractal.
\newblock {\em Probability Theory and Related Fields}, 94(1):1--38, 1992.

\bibitem[Ham00]{Hambly2000}
B.M.\ Hambly.
\newblock On the asymptotics of the eigenvalue counting function for random
  recursive {S}ierpinski gaskets.
\newblock {\em Probability Theory and Related Fields}, 117(2):221--247, 2000.

\bibitem[Hut81]{Hutchinson1981}
J.E.\ Hutchinson.
\newblock Fractals and self-similarity.
\newblock {\em Indiana University Mathematics Journal}, 30(5):713--747, 1981.

\bibitem[Jag75]{Jagers1975}
P.\ Jagers.
\newblock {\em Branching processes with biological applications}.
\newblock Wiley-Interscience [John Wiley \& Sons], London, 1975.
\newblock Wiley Series in Probability and Mathematical Statistics---Applied
  Probability and Statistics.

\bibitem[Kac66]{Kac1966}
M.\ Kac.
\newblock Can one hear the shape of a drum?
\newblock {\em The American Mathematical Monthly}, 73(4, part II):1--23, 1966.

\bibitem[LP93]{LP1993}
M.L.\ Lapidus and C.\ Pomerance.
\newblock The {R}iemann zeta-function and the one-dimensional {W}eyl-{B}erry
  conjecture for fractal drums.
\newblock {\em Proceedings of the London Mathematical Society (3)},
  66(1):41--69, 1993.

\bibitem[LP96]{LP1996}
M.L.\ Lapidus and C.\ Pomerance.
\newblock Counterexamples to the modified {W}eyl-{B}erry conjecture on fractal
  drums.
\newblock {\em Mathematical Proceedings of the Cambridge Philosophical
  Society}, 119(1):167--178, 1996.

\bibitem[LP06]{LP2006}
M.L.\ Lapidus and E.P.J.\ Pearse.
\newblock A tube formula for the {K}och snowflake curve, with applications to
  complex dimensions.
\newblock {\em Journal of the London Mathematical Society. Second Series},
  74(2):397--414, 2006.

\bibitem[LV96]{LV1996}
M.\ Levitin and D.\ Vassiliev.
\newblock Spectral asymptotics, renewal theorem, and the {B}erry conjecture for
  a class of fractals.
\newblock {\em Proceedings of the London Mathematical Society. Third Series},
  72(1):188--214, 1996.

\bibitem[LvF00]{LvF2000}
M.L.\ Lapidus and M.\ van Frankenhuysen.
\newblock {\em Fractal geometry and number theory}.
\newblock Birkh\"auser Boston Inc., Boston, MA, 2000.
\newblock Complex dimensions of fractal strings and zeros of zeta functions.

\bibitem[McM84]{McMullen1984}
C.T.\ McMullen.
\newblock The {H}ausdorff dimension of general {S}ierpi\'nski carpets.
\newblock {\em Nagoya Mathematical Journal}, 96:1--9, 1984.

\bibitem[Mil64]{Milnor1964}
J.\ Milnor.
\newblock Eigenvalues of the {L}aplace operator on certain manifolds.
\newblock {\em Proceedings of the National Academy of Sciences of the United
  States of America}, 51:542, 1964.

\bibitem[Mor46]{Moran1946}
P.A.P.\ Moran.
\newblock Additive functions of intervals and {H}ausdorff measure.
\newblock {\em Proceedings of the Cambridge Philosophical Society}, 42:15--23,
  1946.

\bibitem[MW86]{MW1986}
R.D.\ Mauldin and S.C.\ Williams.
\newblock Random recursive constructions: asymptotic geometric and topological
  properties.
\newblock {\em Transactions of the American Mathematical Society},
  295(1):325--346, 1986.

\bibitem[Ner81]{Nerman1981}
O.~Nerman.
\newblock On the convergence of supercritical general ({C}-{M}-{J}) branching
  processes.
\newblock {\em Zeitschrift f\"{u}r Wahrscheinlichkeitstheorie und verwandte
  Gebiete}, 57(3):365--395, 1981.

\bibitem[Per94]{P1994}
Y.\ Peres.
\newblock The self-affine carpets of {M}c{M}ullen and {B}edford have infinite
  {H}ausdorff measure.
\newblock {\em Mathematical Proceedings of the Cambridge Philosophical
  Society}, 116(3):513--526, 1994.

\bibitem[PS13]{PS2013}
Y.\ Peres and P.\ Sousi.
\newblock Dimension of fractional brownian motion with variable drift.
\newblock {\em arXiv:1310.7002v1}, 2013.

\bibitem[Sim05]{Simon2005}
B.\ Simon.
\newblock {\em Functional integration and quantum physics}.
\newblock AMS Chelsea Publishing, Providence, RI, second edition, 2005.

\bibitem[vdB94]{vdB1994}
M.~van~den Berg.
\newblock Heat content and {B}rownian motion for some regions with a fractal
  boundary.
\newblock {\em Probability Theory and Related Fields}, 100(4):439--456, 1994.

\bibitem[vdBdH99]{vdBdH1999}
M.\ van~den Berg and F.\ den Hollander.
\newblock Asymptotics for the heat content of a planar region with a fractal
  polygonal boundary.
\newblock {\em Proceedings of the London Mathematical Society. Third Series},
  78(3):627--661, 1999.

\bibitem[vdBG98]{vdBG1997}
M.~van~den Berg and P.B.\ Gilkey.
\newblock A comparison estimate for the heat equation with an application to
  the heat content of the {$S$}-adic von {K}och snowflake.
\newblock {\em The Bulletin of the London Mathematical Society},
  30(4):404--412, 1998.

\bibitem[vdBS88]{vdBS1988}
M.~van~den Berg and S.~Srisatkunarajah.
\newblock Heat equation for a region in {${\mathbb{R}}^2$} with a polygonal
  boundary.
\newblock {\em Journal of the London Mathematical Society. Second Series},
  37(1):119--127, 1988.

\bibitem[vdBS90]{vdBS1990}
M.\ van~den Berg and S.\ Srisatkunarajah.
\newblock Heat flow and {B}rownian motion for a region in {$\mathbb{R}^2$} with
  a polygonal boundary.
\newblock {\em Probability Theory and Related Fields}, 86(1):41--52, 1990.

\end{thebibliography}

\end{document}